\documentclass[12pt]{article}
\usepackage{amsmath,amsthm,amssymb}
\usepackage{amssymb,latexsym}

\newtheorem{theorem}{Theorem}

\newtheorem{lemma}{Lemma}

\begin{document}
\baselineskip=17pt

\title{\bf On the distribution of {\boldmath$\alpha p$} modulo one over Piatetski-Shapiro primes}

\author{\bf S. I. Dimitrov}

\date{\bf}
\maketitle

\begin{abstract}
Let $[\, \cdot\,]$ be the floor function and $\|x\|$ denotes the distance from $x$ to the nearest integer.
In this paper we show that  whenever $\alpha$ is irrational and $\beta$ is real
then for any fixed $1<c<12/11$ there exist infinitely many prime numbers $p$ satisfying the inequality
\begin{equation*}
\|\alpha p+\beta\|\ll p^{\frac{11c-12}{26c}}\log^6p
\end{equation*}
and such that $p=[n^c]$.
\quad\\
{\bf  2020 Math.\ Subject Classification}:  11J71 $\cdot$ 11J25 $\cdot$ 11P32  $\cdot$ 11L07
\end{abstract}

\section{Introduction and statement of the result}
\indent

In 1947 Vinogradov \cite{Vinogradov} proved that if $\theta=1/5-\varepsilon$ then there are infinitely many primes $p$ such that
\begin{equation}\label{Vinogradov-inequality}
\|\alpha p+\beta\|<p^{-\theta }\,.
\end{equation}
Subsequently the upper bound for $\theta $ was improved  by several authors and strongest result up to now is due Matom\"{a}ki \cite{Mato} with $\theta=1/3-\varepsilon$ and $\beta=0$.

On the other hand in 1953 Piatetski-Shapiro \cite{Shapiro}
showed that for any fixed $\gamma\in(11/12,1)$ the sequence
\begin{equation*}
\big([n^{1/\gamma}]\big)_{n\in \mathbb{N}}
\end{equation*}
contains infinitely many prime numbers.
The prime numbers of the form $p = [n^{1/\gamma}]$ are called Piatetski-Shapiro primes of type $\gamma$.
Afterwards the interval for $\gamma$ was sharpened many times and the best result
up to now belongs to Rivat and Wu \cite{Rivat-Wu} for $\gamma\in(205/243,1)$.
More precisely they proved that for any fixed $205/243<\gamma<1$ the lower bound
\begin{equation}\label{Louerbound}
\sum\limits_{p\leq X\atop{p=[n^{1/\gamma}]}}1\gg\frac{X^\gamma}{\log X}
\end{equation}
holds.
In order to establish our result we solve Vinogradov's inequality \eqref{Vinogradov-inequality}
with Piatetski-Shapiro primes. Thus we prove the following theorem.
\begin{theorem}\label{Maintheorem}
Let $\gamma$ be fixed with $11/12<\gamma<1$, $\alpha$ is irrational and $\beta$ is real.
Then there exist infinitely many Piatetski-Shapiro prime numbers $p$ of type $\gamma$
such that
\begin{equation*}
\|\alpha p+\beta\|\ll p^{\frac{11-12\gamma}{26}}\log^6p\,.
\end{equation*}
\end{theorem}

\section{Notations}
\indent

Let $C$ is a sufficiently large positive constant.
The letter $p$ will always denote prime number.
The notation $x\sim X$ means that $x$ runs through a subinterval of
$(X, 2X]$, which endpoints are not necessary the same in
the different formulas and may depend on the outer summation variables.
By $[x]$, $\{x\}$ and $\|x\|$ we denote the integer part of $x$,
the fractional part of $x$ and the distance from $x$ to the nearest integer.
Moreover $e(t)$=exp($2\pi it$) and $\psi(t)=\{t\}-1/2$.
As usual $\Lambda(n)$ is von Mangoldt's function and $\tau(n)$ denotes the number of positive divisors of $n$.
Let $\gamma$ be a real constant such that $11/12<\gamma<1$.
Since $\alpha$ is irrational, there are infinitely many different convergents
$a/q$ to its continued fraction, with
\begin{equation}\label{alphaaq}
\bigg|\alpha- \frac{a}{q}\bigg|<\frac{1}{q^2}\,,\quad (a, q) = 1\,,\quad a\neq0
\end{equation}
and $q$ is arbitrary large. Denote
\begin{align}
\label{Nq}
&N=q^\frac{13}{12-6\gamma}\,;\\
\label{Delta}
&\Delta=CN^{\frac{11-12\gamma}{26}}\log^6N\,;\\
\label{HNq}
&H=\left[q^{1/2}\right]\,;\\
\label{Mgamma}
&M=N^{\frac{15-14\gamma}{26}}\,;\\
\label{vgamma}
&v=N^{\frac{29-8\gamma}{52}}\,.
\end{align}

\section{Preliminary lemmas}
\indent

\begin{lemma}\label{Lemma1}
Suppose that $X, Y\geq1$,\, $\big|\alpha-\frac{a}{q}\big|<\frac{1}{q^2}$\,, $(a, q)=1$.
Then
\begin{equation*}
\sum_{n\le X} \, \min \left( Y,\, \frac{1}{ \|\alpha n+\beta\| } \right)
\ll\frac{XY}{q}+Y+(X+q)\log 2q\,.
\end{equation*}
\end{lemma}
\begin{proof}
See (\cite{Vaughan1}, Lemma 1).
\end{proof}

\begin{lemma}\label{Expsumest}
Suppose that $\alpha \in \mathbb{R}$,\, $a \in \mathbb{Z}$,\, $q\in \mathbb{N}$,\,
$\big|\alpha-\frac{a}{q}\big|\leq\frac{1}{q^2}$\,, $(a, q)=1$.
If
\begin{equation}\label{sn1}
\mathcal{S}(X)=\sum\limits_{p\le X}e(\alpha p)\log p
\end{equation}
then
\begin{equation*}
\mathcal{S}(X)\ll \Big(Xq^{-1/2}+X^{4/5}+X^{1/2}q^{1/2}\Big)\log^4X\,.
\end{equation*}
\end{lemma}
\begin{proof}
See (\cite{Iwaniec-Kowalski}, Theorem 13.6).
\end{proof}

\begin{lemma}\label{Expansion}
For any $M\geq2$, we have
\begin{equation*}
\psi(t)=-\sum\limits_{1\leq|m|\leq M}\frac{e(mt)}{2\pi i m}
+\mathcal{O}\Bigg(\min\left(1, \frac{1}{M\|t\|}\right)\Bigg)\,,
\end{equation*}
\begin{proof}
See (\cite{Tolev1}, Lemma 5.2.2).
\end{proof}
\end{lemma}

\begin{lemma}\label{Korput}
Suppose that $f''(t)$ exists, is continuous on $[a,b]$ and satisfies
\begin{equation*}
f''(t)\asymp\lambda\quad(\lambda>0)\quad\mbox{for}\quad t\in[a,b]\,.
\end{equation*}
Then
\begin{equation*}
\bigg|\sum_{a<n\le b}e(f(n))\bigg|\ll(b-a)\lambda^{1/2}+\lambda^{-1/2}\,.
\end{equation*}
\end{lemma}
\begin{proof}
See (\cite{Karatsuba}, Ch.1, Th.5).
\end{proof}

\begin{lemma}\label{Iwaniec-Kowalski}
For any complex numbers $a(n)$ we have
\begin{equation*}
\bigg|\sum_{a<n\le b}a(n)\bigg|^2
\leq\bigg(1+\frac{b-a}{Q}\bigg)\sum_{|q|\leq Q}\bigg(1-\frac{|q|}{Q}\bigg)
\sum_{a<n,\, n+q\leq b}\overline{a(n+q)}a(n),
\end{equation*}
where $Q\geq1$.
\end{lemma}
\begin{proof}
See (\cite{Iwaniec-Kowalski}, Lemma 8.17).
\end{proof}

\section{Proof of the theorem}

\subsection{Outline of the proof}

\indent

Our method goes back to Vaughan \cite{Vaughan1}. We take a periodic with period 1 function such that
\begin{equation}\label{Ffunction}
F_\Delta(\theta)=\begin{cases}
0\quad \mbox{if} \quad -\frac{1}{2}\leq \theta< -\Delta\,,\\
1\quad\mbox{if}\quad-\Delta\leq \theta< \Delta\,,\\
0\quad\mbox{if}\quad\; \Delta\leq \theta<\frac{1}{2}\,.
\end{cases}
\end{equation}
On the basis of \eqref{Delta} and \eqref{Ffunction} we have that
non-trivial lower bound for the sum
\begin{equation*}
\sum\limits_{p\leq N\atop{p=[n^{1/\gamma}]}}F_\Delta(\alpha p+\beta)\log p
\end{equation*}
implies Theorem \ref{Maintheorem}.
For this purpose we define
\begin{equation}\label{Gamma}
\Gamma=\sum\limits_{p\leq N\atop{p=[n^{1/\gamma}]}}\big(F_\Delta(\alpha p+\beta)-2\Delta\big)\log p\,.
\end{equation}

\subsection{Upper bound for $\mathbf{\Gamma}$}
\indent

We shall prove the following fundamental lemma.
\begin{lemma}\label{Mainlemma}
For the sum $\Gamma$ defined by \eqref{Gamma} the upper bound
\begin{equation}\label{Gammaupperbound}
\Gamma\ll N^{\frac{14\gamma+11}{26}}\log^6N
\end{equation}
holds.
\end{lemma}
\begin{proof}

From \eqref{Gamma} we write
\begin{equation}\label{Gammadecomp}
\Gamma=\sum\limits_{p\leq N}\big([-p^\gamma]-[-(p+1)^\gamma]\big)\big(F_\Delta(\alpha p+\beta)-2\Delta\big)\log p
=\Gamma_1+\Gamma_2\,,
\end{equation}
where
\begin{align}
\label{Gamma1}
&\Gamma_1=\sum\limits_{p\leq N}\big((p+1)^\gamma-p^\gamma\big)\big(F_\Delta(\alpha p+\beta)-2\Delta\big)\log p\,,\\
\label{Gamma2}
&\Gamma_2=\sum\limits_{p\leq N}\big(\psi(-(p+1)^\gamma)-\psi(-p^\gamma)\big)\big(F_\Delta(\alpha p+\beta)-2\Delta\big)\log p\,.
\end{align}

\newpage

\textbf{Upper bound for} $\mathbf{\Gamma_1}$
\indent

The function $F_\Delta(\theta)-2\Delta$ is well known to have the expansion
\begin{equation}\label{fDeltaexpansion}
\sum\limits_{1\leq|h|\leq H}\frac{\sin2\pi h\Delta}{\pi h}\,e(h\theta)+\mathcal{O}\Bigg(\min\left(1, \frac{1}{H\|\theta+\Delta\|}\right)+\min\left(1, \frac{1}{H\|\theta-\Delta\|}\right)\Bigg)\,.
\end{equation}
We also have
\begin{equation}\label{p+1}
(p+1)^\gamma-p^\gamma=\gamma p^{\gamma-1}+\mathcal{O}\left(p^{\gamma-2}\right)\,.
\end{equation}
Using \eqref{Gamma1}, \eqref{fDeltaexpansion} and \eqref{p+1} we obtain
\begin{equation}\label{Gamma1est1}
\Gamma_1=\gamma\sum\limits_{p\leq N}p^{\gamma-1}\log p
\sum\limits_{1\leq|h|\leq H}\frac{\sin2\pi h\Delta}{\pi h}\,e\big(h(\alpha p+\beta)\big)+\mathcal{O}\Big(\Sigma\log N\Big)\,,
\end{equation}
where
\begin{equation}\label{Sigma}
\Sigma=\sum\limits_{n=1}^N\Bigg(\min\left(1, \frac{1}{H\|\alpha n+\beta+\Delta\|}\right)
+\min\left(1, \frac{1}{H\|\alpha n+\beta-\Delta\|}\right)\Bigg)\,.
\end{equation}
By \eqref{alphaaq}, \eqref{Nq}, \eqref{HNq}, \eqref{Sigma} and Lemma \ref{Lemma1} we get
\begin{equation}\label{Sigmaest}
\Sigma\ll Nq^{-1} + \frac{N+q}{H} \log N \ll Nq^{-1/2}\log N
\ll N^{\frac{6\gamma+14}{26}}\log N\,.
\end{equation}
Now \eqref{Gamma1est1} and \eqref{Sigmaest} give us 
\begin{equation}\label{Gamma1est2}
\Gamma_1\ll\sum\limits_{h=1}^{H}\min\left(\Delta,\frac{1}{h}\right)\left|\sum\limits_{p\leq N}p^{\gamma-1}e(\alpha h p)\log p \right|+N^{\frac{6\gamma+14}{26}}\log^2N\,.
\end{equation}
Denote
\begin{equation}\label{mathfrakS}
\mathfrak{S}(u)=\sum\limits_{h\leq u}\left|\sum\limits_{p\leq N}p^{\gamma-1}e(\alpha h p)\log p \right|\,.
\end{equation}
Then
\begin{align}\label{Abelsum1}
\sum\limits_{h=1}^{H}\min\left(\Delta,\frac{1}{h}\right)
\left|\sum\limits_{p\leq N}p^{\gamma-1}e(\alpha h p)\log p \right|
&=\frac{\mathfrak{S}(H)}{H}+\int\limits_{\Delta^{-1}}^{H}\frac{\mathfrak{S}(u)}{u^2}\,du\nonumber\\
&\ll(\log N)\max_{\Delta^{-1}\leq u\leq H}\frac{\mathfrak{S}(u)}{u}\,.
\end{align}
On the other hand
\begin{equation}\label{Abelsum2}
\sum\limits_{p\leq N}p^{\gamma-1}e(\alpha h p)\log p=
N^{\gamma-1}S(N)+(1-\gamma)\int\limits_{2}^{N}S(y)y^{\gamma-2}\,dy\,,
\end{equation}
where
\begin{equation}\label{Sy}
S(y)=\sum\limits_{p\leq y}e(\alpha h p)\log p\,.
\end{equation}
From Dirichlet's approximation theorem it follows the existence of integers $a_h$ and $q_h$ such that
\begin{equation}\label{alphahahqh}
\bigg|\alpha h-\frac{a_h}{q_h}\bigg|\leq\frac{1}{q_hq^2}\,,\quad (a_h,q_h)=1\,,\quad 1\leq q_h\leq q^2\,.
\end{equation}
Bearing in mind \eqref{Sy}, \eqref{alphahahqh} and Lemma \ref{Expsumest} we find
\begin{equation}\label{Syest}
S(y)\ll \Big(yq_h^{-1/2}+y^{4/5}+y^{1/2}q_h^{1/2}\Big)\log^4y\,.
\end{equation}
Using \eqref{mathfrakS}, \eqref{Abelsum2} and \eqref{Syest} we deduce
\begin{equation}\label{mathfrakSest1}
\mathfrak{S}(u)\ll N^{\gamma-1}(\log N)^4\sum\limits_{h\leq u}\Big(Nq_h^{-1/2}+N^{4/5}+N^{1/2}q_h^{1/2}\Big)\,.
\end{equation}
Suppose  that
\begin{equation}\label{supposition}
q_h\leq q^{1/3}\,.
\end{equation}
By \eqref{HNq} and \eqref{supposition} we obtain
\begin{equation*}
hq_h\leq Hq^{1/3}\leq q^{5/6}<q\,.
\end{equation*}
From \eqref{alphaaq}, \eqref{alphahahqh} and the last inequality it follows that
\begin{equation}\label{different}
\frac{a}{q}\neq\frac{a_h}{hq_h}\,.
\end{equation}
On the one hand from \eqref{HNq}, \eqref{supposition}  and \eqref{different} we have
\begin{equation}\label{Ontheonehand}
\bigg|\frac{a}{q}-\frac{a_h}{hq_h}\bigg|=
\frac{|ahq_h-qa_h|}{hqq_h}\ge \frac{1}{hqq_h}\geq\frac{1}{Hqq^{1/3}}\geq\frac{1}{q^{11/6}}\,.
\end{equation}
On the other hand by \eqref{alphaaq} and \eqref{alphahahqh} we get
\begin{equation*}
\bigg|\frac{a}{q}-\frac{a_h}{hq_h}\bigg|\le \bigg|\alpha-\frac{a}{q}\bigg|
+\bigg|\alpha-\frac{a_h}{hq_h}\bigg|< \frac{1}{q^2}+\frac{1}{hq_hq^2}\leq\frac{1}{2q^2}\,,
\end{equation*}
which contradicts \eqref{Ontheonehand}.
This rejects the supposition \eqref{supposition}. Therefore
\begin{equation}\label{qhlimits}
q_h\in\big(q^{1/3},q^2\big]\,.
\end{equation}
Taking into account \eqref{Nq}, \eqref{mathfrakSest1} and \eqref{qhlimits} we find
\begin{equation}\label{mathfrakSest2}
\mathfrak{S}(u)\ll uN^{\gamma-1/2}q\log^4N\ll uN^{\frac{14\gamma+11}{26}}\log^4N\,.
\end{equation}
Summarizing \eqref{Gamma1est2}, \eqref{Abelsum1} and \eqref{mathfrakSest2} we deduce
\begin{equation}\label{Gamma1est3}
\Gamma_1\ll N^{\frac{14\gamma+11}{26}}\log^5N\,.
\end{equation}

\textbf{Upper bound for} $\mathbf{\Gamma_2}$
\indent

Using \eqref{Gamma2} and arguing as in $\Gamma_1$ we obtain
\begin{equation}\label{Gamma2est1}
\Gamma_2\ll\sum\limits_{h=1}^{H}\min\left(\Delta,\frac{1}{h}\right)
\left|\sum\limits_{p\leq N}\big(\psi(-(p+1)^\gamma)-\psi(-p^\gamma)\big)e(\alpha h p)\log p \right|
+N^{\frac{6\gamma+14}{26}}\log^2N\,.
\end{equation}
Denote
\begin{equation}\label{Omegau}
\Omega(u)=\sum\limits_{h\leq u}\left|\sum\limits_{p\leq N}
\big(\psi(-(p+1)^\gamma)-\psi(-p^\gamma)\big)e(\alpha h p)\log p \right|\,.
\end{equation}
Then
\begin{align}\label{Abelsum3}
\sum\limits_{h=1}^{H}\min\left(\Delta,\frac{1}{h}\right)
&\left|\sum\limits_{p\leq N}\big(\psi(-(p+1)^\gamma)-\psi(-p^\gamma)\big)e(\alpha h p)\log p \right|\nonumber\\
&=\frac{\Omega(H)}{H}+\int\limits_{\Delta^{-1}}^{H}\frac{\Omega(u)}{u^2}\,du\,.
\end{align}
The estimation \eqref{Gamma2est1} and formula \eqref{Abelsum3} imply
\begin{equation}\label{Gamma2est2}
\Gamma_2\ll(\log N)\max_{\Delta^{-1}\leq u\leq H}\frac{\Omega(u)}{u}+N^{\frac{6\gamma+14}{26}}\log^2N\,.
\end{equation}
We will estimate $\Omega(u)$. From  \eqref{Omegau} and Lemma \ref{Expansion}
with  $M$ defined by  \eqref{Mgamma} it follows
\begin{equation}\label{Omegauest1}
\Omega(u)\ll\big(\Omega_1(u)+u\Xi\big)\log^2N+uN^{1/2}\,,
\end{equation}
where
\begin{align}
\label{Omega1}
&\Omega_1(u)=\sum\limits_{h\leq u}\sum\limits_{m\sim M_1}\frac{1}{m}
\left|\sum\limits_{n\sim N_1}\Lambda(n)e(\alpha h n)
\Big(e\big(-mn^\gamma\big)-e\big(-m(n+1)^\gamma\big)\Big) \right|\,,\\
\label{Xi}
&\Xi=\sum\limits_{n\sim N_1}\min\left(1, \frac{1}{M\|n^\gamma\|}\right)\,,\\
\label{M1N1}
&M_1\leq M/2\,,\quad N_1\leq N/2\,.
\end{align}
Proceeding as in (\cite{Tolev1}, Th. 12.1.1) from \eqref{Xi} and \eqref{M1N1} we get
\begin{equation}\label{Xiest1}
\Xi\ll \Big(NM^{-1}+N^{\gamma/2}M^{1/2}+N^{1-\gamma/2}M^{-1/2}\Big)\log M\,.
\end{equation}
Bearing in mind \eqref{Mgamma} and \eqref{Xiest1}  we find
\begin{equation}\label{Xiest2}
\Xi\ll N^{\frac{14\gamma+11}{26}}\log N\,.
\end{equation}
Next we estimate $\Omega_1(u)$. Replacing
\begin{equation*}
\omega(t)=1-e\big(m(t^\gamma-(t+1)^\gamma)\big)
\end{equation*}
we deduce
\begin{align}\label{Abelsum4}
&\sum\limits_{n\sim N_1}\Lambda(n)e(\alpha h n)
\Big(e\big(-mn^\gamma\big)-e\big(-m(n+1)^\gamma\big)\Big)\nonumber\\
&=\omega(2N_1)\sum\limits_{n\sim N_1}\Lambda(n)e\big(\alpha h n-mn^\gamma\big)\nonumber\\
&-\int\limits_{N_1}^{2N_1}
\left(\sum\limits_{N_1<n\leq t}\Lambda(n)e\big(\alpha h n-mn^\gamma\big)\right)\omega'(t)\,dt\nonumber\\
&\ll m N^{\gamma-1}_1\max_{N_2\in[N_1,2N_1]}|\Theta(N_1,N_2)|\,,
\end{align}
where
\begin{equation}\label{Theta}
\Theta(N_1,N_2)=\sum\limits_{N_1<n\leq N_2}\Lambda(n)e\big(\alpha h n-mn^\gamma\big)\,.
\end{equation}
Now \eqref{Omega1} and \eqref{Abelsum4}  give us
\begin{equation}\label{Omega1est1}
\Omega_1(u)\ll N^{\gamma-1}_1\sum\limits_{h\leq u}\sum\limits_{m\sim M_1}
\max_{N_2\in[N_1,2N_1]}|\Theta(N_1,N_2)|\,.
\end{equation}
Let
\begin{equation}\label{N1less}
N_1\leq N^{\frac{14\gamma-2}{13\gamma}}\,.
\end{equation}
Taking into account \eqref{Mgamma}, \eqref{M1N1}, \eqref{Theta},
\eqref{Omega1est1}  and \eqref{N1less} we obtain
\begin{equation}\label{Omega1est2}
\Omega_1(u)\ll uN^{\frac{14\gamma+11}{26}}\,.
\end{equation}
Hence fort we assume that
\begin{equation}\label{N1greater}
N^{\frac{14\gamma-2}{13\gamma}}<N_1\leq 2N\,.
\end{equation}
We shall find the upper bound of the sum $\Theta(N_1,N_2)$.
Our argument is a modification of (Tolev \cite{Tolev1}, Th. 12.1.1) argument.

Denote
\begin{equation}\label{fdl}
f(d,l)=\alpha hdl-md^\gamma l^\gamma\,.
\end{equation}
Using \eqref{Theta}, \eqref{fdl} and Vaughan's identity (see \cite{Vaughan2}) we get
\begin{equation}\label{Thetadecomp}
\Theta(N_1,N_2)=U_1-U_2-U_3-U_4,
\end{equation}
where
\begin{align}
\label{U1}
&U_1=\sum_{d\le v}\mu(d)\sum_{N_1/d<l\le N_2/d}(\log l)e(f(d,l)),\\
\label{U2}
&U_2=\sum_{d\le v}c(d)\sum_{N_1/d<l\le N_2/d}e(f(d,l)),\\
\label{U3}
&U_3=\sum_{v<d\le v^2}c(d)\sum_{N_1/d<l\le N_2/d}e(f(d,l)),\\
\label{U4}
&U_4= \mathop{\sum\sum}_{\substack{N_1<dl\le N_2 \\d>v,\,l>v}}a(d)\Lambda(l) e(f(d,l))
\end{align}
and where
\begin{equation}\label{cdad}
|c(d)|\leq\log d,\quad  | a(d)|\leq\tau(d)
\end{equation}
and $v$ is defined by \eqref{vgamma}. Consider first $U_2$ defined by \eqref{U2}.
Bearing in mind  \eqref{fdl} we find
\begin{equation}\label{f''ll}
|f^{\prime\prime}_{ll}(d,l)|\asymp m d^2N_1^{\gamma-2}\,.
\end{equation}
Now \eqref{f''ll} and Lemma \ref{Korput} give us
\begin{equation}\label{sumefdl}
\sum\limits_{N_1/d<l\le N_2/d}e(f(d, l))\ll m^{1/2}N_1^{\gamma/2}+m^{-1/2}d^{-1}N_1^{1-\gamma/2}\,.
\end{equation}
From  \eqref{Mgamma}, \eqref{vgamma}, \eqref{U2}, \eqref{cdad} and \eqref{sumefdl} it follows
\begin{equation}\label{U2est}
U_2\ll\big(N_1^{\gamma/2}m^{1/2}v+m^{-1/2}N_1^{1-\gamma/2}\big)\log^2N\ll N_1^{\gamma/2}m^{1/2}v\log^2N\,.
\end{equation}
In order to estimate $U_1$ defined by \eqref{U1} we apply Abel's summation formula.
Then arguing as in the estimation of $U_2$  we obtain
\begin{equation}\label{U1est}
U_1\ll N_1^{\gamma/2}m^{1/2}v\log^2N\,.
\end{equation}
Next we consider $U_3$ and $U_4$  defined by \eqref{U3} and \eqref{U4}.
We have
\begin{equation}\label{U3U'3}
U_3\ll|U'_3|\log N\,,
\end{equation}
where
\begin{equation}\label{U'3}
U'_3=\sum_{D<d\le 2D}c(d)\sum_{L<l\le 2L\atop{N_1<dl\le N_2}}e(f(d,l))
\end{equation}
and where
\begin{equation}\label{ParU'3}
N_1/4\leq DL\le 2N_1\,, \quad  v/2\leq D\leq v^2\,.
\end{equation}
Also
\begin{equation}\label{U4U'4}
U_4\ll|U'_4|\log N\,,
\end{equation}
where
\begin{equation}\label{U'4}
U'_4=\sum_{D<d\le 2D}a(d)\sum_{L<l\le 2L\atop{N_1<dl\le N_2}}\Lambda(l)e(f(d,l))
\end{equation}
and where
\begin{equation}\label{ParU'4}
N_1/4\leq DL\le 2N_1\,, \quad  v/2\leq D\leq  2N_1/v\,.
\end{equation}
By  \eqref{vgamma}, \eqref{N1greater}, \eqref{ParU'3} and \eqref{ParU'4} it follows that
the conditions for the sum $U'_4$ are more restrictive than the conditions for the sum $U'_3$.
Bearing in mind this consideration and the coefficients of $U'_3$ and $U'_4$ we make a conclusion
that it's enough to estimate the sum $U'_4$ with the conditions
\begin{equation}\label{ParU'4new}
N_1/4\leq DL\le 2N_1\,, \quad  N^\frac{1}{2}_1/2\leq D\leq v^2\,.
\end{equation}
From \eqref{cdad}, \eqref{U'4}, \eqref{ParU'4new} and Cauchy's inequality we deduce
\begin{align}\label{U'42est1}
|U'_4|^2&\ll \sum_{D<d\le 2D}\tau^2(d)\sum_{D<d\le 2D}\bigg|\sum_{L_1<l\le L_2}\Lambda(l)e(f(d,l))\bigg|^2\nonumber\\
&\ll D(\log N)^3\sum_{D<d\le 2D}\bigg|\sum_{L_1<l\le L_2}\Lambda(l)e(f(d,l))\bigg|^2,
\end{align}
where
\begin{equation}\label{maxmin1}
L_1=\max{\bigg\{L,\frac{N_1}{d}\bigg\}},\quad L_2=\min{\bigg\{2L, \frac{N_2}{d}\bigg\}}\,.
\end{equation}
Using \eqref{ParU'4new} -- \eqref{maxmin1}  and Lemma \ref{Iwaniec-Kowalski} with $Q\leq L/2$ we find
\begin{align}\label{U'42est2}
|U'_4|^2&\ll D(\log N)^3   \sum_{D<d\le 2D}\frac{L}{Q}
\sum_{|q|\leq Q}\bigg(1-\frac{|q|}{Q}\bigg)
\sum_{L_1<l\le L_2\atop{L_1<l+q\le L_2}}\Lambda(l+q)\Lambda(l)e\big(f(d,l)-f(d,l+q)\big)\nonumber\\
&\ll \Bigg(\frac{LD}{Q}\sum_{0<|q|\leq Q}
\sum_{L<l\le 2L\atop{L<l+q\le 2L}}\Lambda(l+q)\Lambda(l)
\bigg|\sum_{D_1<d\le D_2}e\big(g_{l,q}(d)\big)\bigg|\nonumber\\
&\quad\quad\quad\quad\quad\quad\quad\quad\quad
\quad\quad\quad\quad\quad\quad\quad\quad\quad+\frac{(LD)^2}{Q}\log N\Bigg)\log^3N\,,
\end{align}
where
\begin{equation}\label{maxmin2}
D_1=\max{\bigg\{D,\frac{N_1}{l},\frac{N_1}{l+q}\bigg\}},\quad
D_2=\min{\bigg\{2D,\frac{N_2}{l},\frac{N_2}{l+q}\bigg\}}
\end{equation}
and
\begin{equation}\label{gd}
g(d)=g_{l,q}(d)=f(d,l)-f(d,l+q)\,.
\end{equation}
It is easy to see that the sum over negative $q$ in formula \eqref{U'42est2}
is equal to the sum over positive $q$. Therefore
\begin{align}\label{U'42est3}
|U'_4|^2\ll\Bigg(\frac{LD}{Q}\sum_{1\leq q\leq Q}
\sum_{L<l\le 2L}\Lambda(l+q)\Lambda(l)
\bigg|&\sum_{D_1<d\le D_2}e(g_{l,q}(d))\bigg|\nonumber\\
&\quad\quad+\frac{ (LD)^2}{Q}\log N\Bigg)\log^3N\,.
\end{align}
Consider the function $g(d)$.
From   \eqref{fdl} and \eqref{gd} we obtain
\begin{equation}\label{g''d}
|g''(d)|\asymp m D^{\gamma-2} |q| L^{\gamma-1}\,.
\end{equation}
Taking into account  \eqref{maxmin2}, \eqref{g''d} and Lemma \ref{Korput} we get
\begin{equation}\label{sumegd}
\sum\limits_{D_1<d\leq D_2}e(g(d))
\ll m^{1/2}q^{1/2}D^{\gamma/2}L^{\gamma/2-1/2}+m^{-1/2}q^{-1/2}D^{1-\gamma/2}L^{1/2-\gamma/2}\,.
\end{equation}
We choose
\begin{equation}\label{Q0mDL}
Q=\min\big(L/4\,, Q_0\big)\,,
\end{equation}
where
\begin{equation*}
Q_0=m^{-1/3}D^{2/3-\gamma/3}L^{1/3-\gamma/3}\,.
\end{equation*}
Here \eqref{Mgamma}, \eqref{M1N1}, \eqref{N1greater}, \eqref{ParU'4new} and  the direct verification assure us that
\begin{equation*}
Q_0>N^{\frac{115}{858}}\,.
\end{equation*}
By \eqref{U'42est3}, \eqref{sumegd}  and \eqref{Q0mDL} we deduce
\begin{align}\label{U'42est4}
|U'_4|^2&\ll \big( D^2L^2Q^{-1}+m^{1/2}Q^{1/2}D^{1+\gamma/2}L^{3/2+\gamma/2}+m^{-1/2}Q^{-1/2}D^{2-\gamma/2}L^{5/2-\gamma/2}\big)\log^4N\nonumber\\
&\ll \Big(D^2L^2L^{-1}+ D^2L^2Q_0^{-1}+m^{1/2}Q_0^{1/2}D^{1+\gamma/2}L^{3/2+\gamma/2}\nonumber\\
&\hspace{48mm}+m^{-1/2}D^{2-\gamma/2}L^{5/2-\gamma/2}\big(L^{-1/2}+Q_0^{-1/2}\big)\Big)\log^4N\nonumber\\
&\ll \big( D^2L+m^{1/3}D^{4/3+\gamma/3}L^{5/3+\gamma/3}+m^{-1/2}D^{2-\gamma/2}L^{2-\gamma/2}\nonumber\\
&\hspace{60mm}+m^{-1/3}D^{5/3-\gamma/3}L^{7/3-\gamma/3}\big)\log^4N\,.
\end{align}
From \eqref{ParU'4new} and \eqref{U'42est4} it follows
\begin{equation}\label{U'4est}
|U'_4|\ll \big( N_1^{1/2}v+M^{1/6}N_1^{3/4+\gamma/6}\big)\log^2N\,.
\end{equation}
Now  \eqref{U4U'4} and \eqref{U'4est} imply
\begin{equation}\label{U4est}
U_4\ll \big( N_1^{1/2}v+M^{1/6}N_1^{3/4+\gamma/6}\big)\log^3N\,.
\end{equation}
Arguing as in the estimation of $U'_4$ for the sum \eqref{U'3} we find
\begin{equation}\label{U'3est}
|U'_3|\ll \big( N_1^{1/2}v+M^{1/6}N_1^{3/4+\gamma/6}\big)\log^2N\,.
\end{equation}
The estimates \eqref{U3U'3} and \eqref{U'3est} give us
\begin{equation}\label{U3est}
U_3\ll \big( N_1^{1/2}v+M^{1/6}N_1^{3/4+\gamma/6}\big)\log^3N\,.
\end{equation}
Summarizing \eqref{Thetadecomp}, \eqref{U2est}, \eqref{U1est},
\eqref{U4est} and \eqref{U3est} we get
\begin{equation}\label{Thetaest}
\Theta(N_1,N_2)\ll \Big( N_1^{1/2}v+M^{1/6}N_1^{3/4+\gamma/6}+N_1^{\gamma/2}m^{1/2}v\Big)\log^3N\,.
\end{equation}
By \eqref{Mgamma}, \eqref{vgamma}, \eqref{Omega1est1}, \eqref{N1greater} and \eqref{Thetaest}  it follows
\begin{equation}\label{Omega1est3}
\Omega_1(u)\ll uN^{\frac{14\gamma+11}{26}}\log^3N\,.
\end{equation}
From \eqref{Gamma2est2}, \eqref{Omegauest1},
\eqref{Xiest2}, \eqref{Omega1est2} and \eqref{Omega1est3} we deduce
\begin{equation}\label{Gamma2est3}
\Gamma_2\ll N^{\frac{14\gamma+11}{26}}\log^6N\,.
\end{equation}
Using \eqref{Gammadecomp}, \eqref{Gamma1est3} and \eqref{Gamma2est3}
we establish the upper bound \eqref{Gammaupperbound}.

The lemma is proved.
\end{proof}

\subsection{The end of the proof}
\indent

Bearing in mind  \eqref{Louerbound}, \eqref{Delta},
\eqref{Gamma} and Lemma \ref{Mainlemma} we obtain
\begin{equation*}
\sum\limits_{p\leq N\atop{p=[n^{1/\gamma}]}}F_\Delta(\alpha p+\beta)\log p\gg N^{\frac{14\gamma+11}{26}}\log^6N\,.
\end{equation*}

This completes the proof of theorem.

\vskip20pt
\footnotesize
\begin{flushleft}
S. I. Dimitrov\\
\quad\\
Faculty of Applied Mathematics and Informatics\\
Technical University of Sofia \\
Blvd. St.Kliment Ohridski 8 \\
Sofia 1756, Bulgaria\\
e-mail: sdimitrov@tu-sofia.bg\\
\end{flushleft}

\end{document}